\documentclass[final,notitlepage,12pt,reqno]{amsart}
\usepackage{graphicx}
\usepackage{calc}
\usepackage{url}
\usepackage{mathrsfs}
\usepackage{rotating}
\newlength{\depthofsumsign}
\makeatletter
\let\I\@undefined
\makeatother
\usepackage{geometry}
\usepackage{indentfirst}
\usepackage[normalem]{ulem}
\usepackage{float}
\usepackage{amsthm}

\usepackage{enumitem}[2011/09/28]
\setenumerate{align=left, leftmargin=0pt,labelsep=.5em, labelindent=0\parindent,listparindent=\parindent,itemindent=*}
\usepackage{array}
\usepackage{empheq}
\usepackage{natbib}
\usepackage[all]{xy}

\usepackage{caption}[2012/02/19]

\usepackage{longtable,lscape}

\usepackage{fouriernc}
\usepackage{fourier}
\usepackage{amssymb,bm,amsmath}
\usepackage{amsfonts}

\usepackage[OT2,T1,T2A]{fontenc}

\usepackage[english,french,german,russian]{babel}
\usepackage{appendix}

\DeclareMathOperator{\IKM}{\mathbf{IKM}}

\DeclareMathOperator{\ParH}{\mathsf{Par}}
\DeclareMathOperator{\HPB}{\mathsf{HPB}}
\DeclareMathOperator{\HT}{\widehat{\mathscr H}}

\DeclareMathOperator{\sgn}{sgn}

\DeclareMathOperator{\D}{d}
\DeclareMathOperator{\I}{Im}

\bibpunct{[}{]}{,}{n}{}{;}

\def\eor{\hfill$ \square$}

\theoremstyle{plain}
\newtheorem{theorem}{Theorem}[section]
\newtheorem{proposition}[theorem]{Proposition}
\newtheorem{lemma}[theorem]{Lemma}

\newtheorem{conjecture}[theorem]{Conjecture}

\newenvironment{remark}[1][Remark]{\begin{trivlist}
\item[\hskip \labelsep {\bfseries #1}]}{\end{trivlist}}

\theoremstyle{definition}
\newtheorem{definition}[theorem]{Definition}

\numberwithin{equation}{section}

\usepackage{color}

\begin{document}

%\pagenumbering{roman}
\selectlanguage{english}
%+Title
\title{Hilbert  Transforms and Sum Rules of Bessel Moments}
\author{Yajun Zhou}
\address{Program in Applied and Computational Mathematics (PACM), Princeton University, Princeton, NJ 08544; Academy of Advanced Interdisciplinary Sciences (AAIS), Peking University, Beijing 100871, P. R. China }
\email{yajunz@math.princeton.edu, yajun.zhou.1982@pku.edu.cn}\thanks{\indent * This research was supported in part  by the Applied Mathematics Program within the Department of Energy
(DOE) Office of Advanced Scientific Computing Research (ASCR) as part of the Collaboratory on
Mathematics for Mesoscopic Modeling of Materials (CM4)}
\date{\today}

\maketitle
%-Title

%+Abstract
\begin{center}\footnotesize{\textit{To the memory of Jonathan M. Borwein }(1951--2016)}\end{center}\vspace{-1em}
\begin{abstract}
    Using Hilbert transforms, we establish two families of sum rules   involving    Bessel moments, which are  integrals  associated with Feynman diagrams in two-dimensional quantum field theory. With these linear relations among Bessel moments, we  verify and generalize two conjectures by Bailey--Borwein--Broadhurst--Glasser and Broadhurst--Mellit. \end{abstract}
%-Abstract

%+Contents
%\pagenumbering{roman}
%\tableofcontents

%-Contents
%\clearpage

\pagenumbering{arabic}

\section{Introduction}
%\subsection{Background and motivations}
Let $I_0$ and $K_0$ be the modified Bessel functions of zeroth order, defined through the following Schl\"afli integral representations \cite[][\S6.15]{Watson1944Bessel}:\begin{align} I_0(t)=\frac{1}{\pi}\int_0^\pi e^{t\cos\theta}\D\theta,\quad K_0(t)=\int_0^\infty e^{-t\cosh u}\D u.\end{align}Bessel moments,  in the form of absolutely convergent integrals  \begin{align} \IKM(a,b;c):=\int_0^\infty[I_0(t)]^a[K_0(t)]^{b}t^{c}\D t\end{align} for certain non-negative integers $a,b,c$, arise naturally from perturbative diagrammatic expansions in two-dimensional quantum field theory \cite{Groote2007,Bailey2008,Broadhurst2016,BroadhurstMellit2016}.

The   Bessel moments
$\IKM(a,b;c)$ involving no more than four Bessel functions (that is, $a+b\leq4$) have been thoroughly studied  \cite{Bailey2008}: arithmetically speaking, the  $ \IKM(a,b;1)$ values are expressible as rational multiples of  special $L$-values when $a+b\leq4$; combinatorially speaking,    generating functions for the corresponding sequences $ \IKM(a,b;c)$  with respect to $c\in\mathbb Z_{>0}$ are   explicitly known. In contrast, for  $a+b\geq5$, the combinatorial structure of Bessel moments $ \IKM(a,b;c)$   and their relation to  $L$-functions largely remain elusive. For the $a+b=5$ case alone, closed-form evaluations of $ \IKM(2,3;1)$ and $ \IKM(1,4;1)$  have drawn heavily on  critical advances in symbolic computation \cite{Bailey2008}, algebraic geometry \cite{BlochKerrVanhove2015}
and number theory \cite{RogersWanZucker2015,Samart2016} within the last decade. A recent progress towards the understanding of Bessel moments with $a+b=8,c=1$ (see \cite[][\S7.6]{Broadhurst2016} and \cite[][\S7]{BroadhurstMellit2016})
 has benefited from Yun's insights  \cite{Yun2015} into the Langlands program.

On a different note, numerical experimentations seem to support the existence of various algebraic relations among Bessel moments with   $a+b\geq5$. As examples of sum rules  suggested from high-precision numerical computations, we mention  here two open problems concerning Bessel moments: a cancelation formula   proposed by Bailey--Borwein--Broadhurst--Glasser in 2008, and  a conjecture on integrality  formulated by Broadhurst (jointly with Mellit, and in honor of Crandall) in 2016.

\begin{conjecture}[B$^3$G sum rule {\cite[][``final conjecture'', (220)]{Bailey2008}}]\label{conj:BBBG}For each pair of integers $ (n,k)\in\mathbb Z^2$ satisfying $n\geq2k\geq2$, the following sum of integrals vanishes identically:\begin{align}
Z_{2n,n-2k}:={}&\sum_{m=0}^{\lfloor n/2\rfloor}(-1)^m{n\choose{2m}}\int_0^\infty [\pi I_0(t)]^{n-2m}[K_0(t)]^{n+2m}t^{n-2k}\D t=0,
\end{align}where ${n\choose j}=\frac{n!}{j!(n-j)!}$ and $ \lfloor x\rfloor$ represents the greatest integer less than or equal to $x$.\end{conjecture}\begin{conjecture}[Broadhurst--Mellit  integer sequence {\cite[][(149) in Conjecture 5]{Broadhurst2016}}, \textit{viz.}\ Crandall numbers]\label{conj:B_int_seq}For each $n\in\mathbb Z_{>0}$, the following integral\begin{align}
A(n):={}&\left(\frac{2}{\pi}\right)^4
\int_0^\infty\left\{[\pi I_0(t)]^2 - [K_0(t)]^2\right\} I_0(t)[K_0(t)]^5\,(2t)^{2n-1} \D t
\label{eq:An}
\end{align}evaluates to an integer.\end{conjecture}

%\subsection{Statement of results and plan of proof}
 In this  note, we harness the  Hilbert transform to verify both Conjectures \ref{conj:BBBG} and \ref{conj:B_int_seq},   in a unified and coherent framework, without explicitly evaluating  individual Bessel moments contained in the sum rules.

  The article is organized as follows. In \S\ref{sec:HT_Bessel}, we start from a brief overview of the classical  Hilbert transform (\S\ref{subsec:HT_basics}), a major analytic and algebraic tool in this work; we then compute Hilbert transforms    for some simple algebraic expressions involving modified Bessel functions (\S\ref{subsec:HT_I_K}), paving the way for the subsequent proofs of Conjecture \ref{conj:BBBG} in \S\ref{subsec:BBBG}, and Conjecture \ref{conj:B_int_seq} in \S\ref{subsec:B_int_seq}. In addition to  proving both these conjectures, we construct a conjugate to the B$^3$G sum rule in \S\ref{subsec:BBBG}:\begin{align}
Y_{2n,n-2k}:=\sum_{m=1}^{\lfloor n/2\rfloor+1}(-1)^{m}{n\choose{2m-1}}\int_0^\infty [\pi I_0(t)]^{n-2m+1}[K_0(t)]^{n+2m-1}t^{n-2k-1}\D t=0,
\end{align}for $ n-1\geq2k\ge2$,  and obtain  generalizations of the Broadhurst--Mellit integer sequence in \S\ref{subsec:B_int_seq}:\begin{align}
\frac{4}{\pi^{2m+1}}\sum_{\ell=1}^{m+1}(-1)^{\ell-1} {2m\choose{2\ell-1}}\int_0^\infty [\pi I_0(t)]^{2(m-\ell)+1}[K_0(t)]^{2(m+\ell)-1}(2t)^{2(n+m)-3}\D t\in{}&\mathbb Z_{>0},\\\frac{4^{n}}{\pi^{2m}}\sum_{\ell=1}^{m}(-1)^{\ell-1} {2m-1\choose{2\ell-1}}\int_0^\infty [\pi I_0(t)]^{2(m-\ell)}[K_0(t)]^{2(m+\ell-1)}(2t)^{2(n+m-2)}\D  t\in{}&\mathbb Z_{>0,}
\end{align} for all $m,n\in\mathbb Z_{>0}$.

 \section{Hilbert transforms involving modified Bessel functions\label{sec:HT_Bessel}}\subsection{Basic properties of Hilbert transforms\label{subsec:HT_basics}}The key device in our proof is the Hilbert transform  $\HT $, which   operates on a suitably regular function $ f(x),\text{a.e.}\,x\in\mathbb R$ through a Cauchy principal value:\footnote{For our purposes, it suffices to evaluate the Hilbert transform of a function for almost every (a.e.)\ real variable, leaving out a subset of measure zero that will not affect subsequent computations of Lebesgue integrals. Later afterwards, an equal sign may also be used to denote an equality that is valid almost everywhere (even when ``a.e.'' is not written), depending on context. }\begin{align}(\HT f)(x)=\mathscr P\int_{-\infty}^\infty\frac{f(\xi)\D \xi}{\pi(x-\xi)},\quad \text{a.e.}\,x\in\mathbb R .\label{eq:HT_defn}\end{align}When restricted to an $ L^p$ space for $ 1<p<\infty$,   the Hilbert transform induces a bounded linear operator  $ \HT :L^p(\mathbb R)\longrightarrow L^p(\mathbb R)$ \cite[][p.~188]{SteinWeiss}.

We list here three fundamental properties of the Hilbert transform on $\mathbb R$:\begin{enumerate}[leftmargin=*,  label=(HT\arabic*),ref=(HT\arabic*),
widest=a, align=left]\item \label{itm:HT-1}\textbf{(Parseval-type Identity) }For  $ f\in L^p(\mathbb R),p>1$; $ g\in L^q(\mathbb R),q>1$ and $ \frac1p+\frac1q=1$, we have  \cite[][(10.81)]{KingVol1}:\begin{align}\langle f,\HT g\rangle+\langle g,\HT f\rangle=0,\quad\text{where }\langle f_1,f_2\rangle:=\int_{-\infty}^\infty f_1(x)f_2(x)\D x.\label{eq:Tricomi_Parseval1}\end{align}
\item \label{itm:HT-2}\textbf{(Hardy--Poincar\'e--Bertrand Formula)} For  $ f\in L^p(\mathbb R),p>1$; $ g\in L^q(\mathbb R),q>1$ and $ \frac1p+\frac1q=1$, we have \cite[][(2.227)]{KingVol1}:\begin{align}
\HT (f\HT g+g\HT f)=(\HT f)(\HT g)-fg.\label{eq:HPB}
\end{align} \item \label{itm:HT-3}\textbf{(Moment Formula)} For $n\in\mathbb Z_{>0}$, the following identity  \cite[][(4.113)]{KingVol1}\begin{align}
\mathscr P\int_{-\infty}^\infty\frac{\xi ^{n}f(\xi)\D \xi}{\pi(x-\xi)}=x^n\mathscr P\int_{-\infty}^\infty\frac{f(\xi)\D \xi}{\pi(x-\xi)}-\frac{1}{\pi}\sum_{k=0}^{n-1}  x^k\int_{-\infty}^\infty \xi^{n-1-k}f(\xi)\D\xi\label{eq:HT_pow}
\end{align}holds when each term therein is well-defined.\end{enumerate}Due to frequent invocations of properties \ref{itm:HT-1} and \ref{itm:HT-2}, we will abbreviate the relations in \eqref{eq:Tricomi_Parseval1} and \eqref{eq:HPB} as $ \ParH(f,g)$ and $ \HPB(f,g)$, respectively.

As a  consequence of the Hardy--Poincar\'e--Bertrand formula, the operator $ i\HT $ is its own inverse \cite[][(4.18)]{KingVol1}. In other words, \begin{align} \HT ^2=-\widehat I:L^p(\mathbb R)\longrightarrow L^p(\mathbb R),\quad 1<p<\infty\label{eq:HT_inv}\end{align}maps an $L^p$ function $f(x),x\in\mathbb R$ to $-f(x),x\in\mathbb R$.

\subsection{Applications to  modified Bessel functions\label{subsec:HT_I_K}}
To simplify notations, we introduce  short-hands for certain expressions involving the modified Bessel functions.\begin{definition}For a real variable $x\in\mathbb R$, we define the functions $ \iota,\iota_+,\iota_-,$ and $\kappa,\kappa_+,\kappa_-$ as follows:\begin{align}\iota(x):={}&\pi I_0(x);&\kappa(x):={}&K_0(|x|);\\\iota_+(x):={}& \iota(x)e^{- x}\mathbb1_{(0,\infty)}(x);&\kappa_+(x):={}&\kappa(x)e^{-x};\\\iota_-(x):={}&\iota(x)e^x\mathbb1_{(-\infty,0)}(x);&\kappa_-(x):={}&\kappa(x)e^{x},\end{align}where the indicator function behaves like $ \mathbb1_A(x)=1,x\in A;\mathbb1_A(x)=0,x\notin A$. It is clear that $ \kappa,\iota\kappa\in L^p(\mathbb R)$ for $ p>1$ and  $ \iota_+,\iota_-,\kappa_+,\kappa_-\in L^p(\mathbb R)$ for $p>2$. \end{definition}\begin{lemma}[Hilbert transform of modified Bessel functions]\label{lm:HT_I_K}We have the following Hilbert transform formulae:\begin{align}
\HT \iota_+={}&-\kappa_+,&
\HT \iota_-={}&\phantom{+\,\,}\kappa_- ,\label{eq:HT_I}\\\HT \kappa_+={}&\phantom{ +\,\,}\iota_+,&\HT \kappa_-={}&-\iota_-,
\label{eq:HT_K}\end{align}which entail the following Feynman rules:\begin{align}
\HT (\iota\kappa\sgn)={}&-\kappa^{2},\label{eq:HT_IKsgn}\\\HT (\kappa^2)={}& \iota\kappa\sgn,\label{eq:HT_KK}
\end{align}for $\sgn(x)=x/|x|,\forall x\in \mathbb R\smallsetminus\{0\}$.\end{lemma}\begin{proof}The results in \eqref{eq:HT_I} and \eqref{eq:HT_K} are classical (see \cite[][\S9.10]{KingVol1} or \cite[][Appendix 1, Tables 1.8H and 1.8I]{KingVol2}).
Furthermore,
we note that  \eqref{eq:HT_K}
follows from  \eqref{eq:HT_I} and \eqref{eq:HT_inv}.

Specializing the Hardy--Poincar\'e--Bertrand formula to $ \HPB(\iota_+,\iota_{-}\mathbb1_{[-L,L]})$ where the values of $ \iota_{-}\mathbb1_{[-L,L]}$ are zero-padded outside a  bounded interval $ [-L,L]$, and exploiting the continuity of the Hilbert transform as a linear operator,  we obtain the following formula in the $ L\to\infty$ limit: \begin{align}
\HT (\iota_+\kappa_{-}-\iota_{-}\kappa_{+})=-\kappa_+\kappa_{-}-\iota_+\iota_-.
\end{align} Here, the left-hand side of the equation above is equal to $ \HT (\iota\kappa\sgn)$, while the right-hand side simplifies to $-\kappa^2 $. This proves  \eqref{eq:HT_IKsgn}. Applying $\HT $ to both sides of   \eqref{eq:HT_IKsgn}, we arrive at \eqref{eq:HT_KK}.  \end{proof}\begin{remark}With a Parseval-type identity $ \ParH(\iota\kappa\sgn,\kappa^2)$, we immediately deduce from  \eqref{eq:HT_IKsgn} and \eqref{eq:HT_KK} a sum rule for Bessel moments \cite[][(91) and (219)]{Bailey2008}\begin{align}
0={}&\frac{\langle  \iota\kappa\sgn,\HT (\kappa^2)\rangle+\langle\kappa^2,\HT (  \iota\kappa\sgn)\rangle}{2}\notag\\={}&\int_0^\infty[\pi I_0(t)]^2[K_0(t)]^2\D t-\int_0^\infty  [K_0(t)]^4\D t=:Z_{4,0}.\label{eq:Z_4_0}
\end{align}Unlike the analysis in  \cite[][(91)]{Bailey2008}, during our proof of the cancelation formula above, we have not explicitly computed either   $\int_0^\infty[\pi I_0(t)]^2[K_0(t)]^2\D t$ or $ \int_0^\infty  [K_0(t)]^4\D t$.\eor\end{remark}

\section{Applications to  sum rules of Bessel moments}
\subsection{B$^3$G sum rule\label{subsec:BBBG} }
To the best of our knowledge,  except for the case of $ Z_{4,0}$ (which was solved by brute force \cite[][(91)]{Bailey2008} and reaffirmed in \S\ref{subsec:HT_I_K} above), Conjecture \ref{conj:BBBG} has remained hitherto open. Both \cite[][(221)]{Bailey2008} and \cite[][(112)]{Broadhurst2016} have singled out the verification of $ Z_{6,1}=0$ as an unsolved problem.\footnote{Confusingly, in footnote 13 of \cite{Bailey2008}, the authors apparently declared that the conjecture on $ Z_{2n,n-2k}=0$ ``has also now been resolved'', without supplying further citations to publications or preprints. This contradicts the recent claims in  \cite[][(112)]{Broadhurst2016} and \cite[][(5.6)]{BroadhurstMellit2016} on the open status of $ Z_{6,1}=0$.}

Before handling  Conjecture \ref{conj:BBBG} in full, we first walk through the proof of  $ Z_{6,1}=0$ and $ Z_{8,0}=0$ in the next  lemma, to illustrate our strategies. In what follows, for $n\in\mathbb Z_{>0}$, we set $\varpi_n(x):= x^n,\forall x\in\mathbb R$; we further define $\varpi_0(x):=1,\forall x\in\mathbb R$.
\begin{lemma}We have $Z_{6,1}=0$ and  $ Z_{8,0}=0$.\end{lemma}\begin{proof}
By $\HPB(\iota\kappa\sgn,\kappa^{2})$, we have a Hilbert transform formula:\begin{align}
\HT(\iota^2\kappa^2-\kappa^{4})={}&-2\iota \kappa^3\sgn.\label{eq:HT_IIKK_KKKK}
\end{align}Appealing to \eqref{eq:HT_pow}, we   deduce \begin{align}
\HT((\iota^2\kappa^2-\kappa^{4})\varpi_{1})={}&-2\iota \kappa^3\varpi_{1}\sgn-\frac{1}{\pi}\int_{-\infty}^\infty (\iota^2\kappa^2-\kappa^{4})(\xi)\D\xi=-2\iota \kappa^3\varpi_{1}\sgn,\label{eq:HT_IIKK_KKKK_p1}
\end{align}upon invoking $Z_{4,0}=0$ in the last step. By $\ParH(\kappa^2,(\iota^2\kappa^2-\kappa^{4})\varpi_{1})$, we arrive at \begin{align}
0={}&\frac{\langle\kappa^2,-2\iota \kappa^3\varpi_{1}\sgn\rangle+\langle \iota\kappa\sgn,(\iota^2\kappa^2-\kappa^{4})\varpi_{1}\rangle}{2}\notag\\={}&\int_0^\infty[\pi I_0(t)]^{3}[K_0(t)]^3t\D t-3\int_0^\infty\pi I_0(t)[K_0(t)]^5t\D t=Z_{6,1},\label{eq:Z61_pf}
\end{align}as claimed.

By \eqref{eq:HT_IIKK_KKKK} and $\ParH(\iota^2\kappa^2-\kappa^{4},2\iota \kappa^3\sgn)$, we have \begin{align}
0= {}&\frac{ \langle\iota^2\kappa^2-\kappa^{4},\HT(2\iota \kappa^3\sgn)\rangle+\langle2\iota \kappa^3\sgn,\HT(\iota^2\kappa^2-\kappa^{4})\rangle}{2}\notag\\={}&\int_0^\infty[\pi I_0(t)]^{4}[K_0(t)]^4\D t-6\int_0^\infty[\pi I_0(t)]^{2}[K_0(t)]^6\D t+\int_0^\infty[K_0(t)]^8\D t=Z_{8,0},\label{eq:Z80_pf}
\end{align}as stated.\end{proof}

In the arguments above, we have established the sum rules $Z_{6,1}=0$ and $Z_{8,0}=0$ on a simpler vanishing identity $Z_{4,0}=0$.
In the next  proposition, we proceed to more general transitions from $2\ell$ Bessel factors to $2\ell+2$ and $2\ell+4$ scenarios, where $\ell\in\mathbb Z_{>0}$.
\begin{proposition}[Hilbert ladders]\label{prop:Hilbert_ladders}We have the following Feynman rules for all  $\ell\in\mathbb Z_{>0}$: \begin{align}
\begin{cases}
\HT(\zeta_\ell)=\phantom{+}\eta_\ell, \\\HT(\eta_\ell)=-\zeta_\ell,\label{eq:H_zeta_ell_eta_ell}
\end{cases}
\end{align}where the Hilbert ladders are defined by\begin{align}\quad \text{} \quad\begin{cases}2\zeta_{\ell}=\phantom{i}\kappa^{\ell}[(\iota\sgn+i\kappa)^\ell+(\iota\sgn-i\kappa)^\ell], \\
2\eta_{\ell}=i\kappa^{\ell}[(\iota\sgn+i\kappa)^\ell-(\iota\sgn-i\kappa)^\ell]. \\
\end{cases}
\quad \label{eq:Hilb_ladders_zeta_eta}
\end{align} Moreover, we have  \begin{align}\HT(\zeta_\ell \varpi_{ j})=\eta_\ell \varpi_j, \quad\forall j\in\mathbb Z\cap[0,\ell).\label{eq:zeta_ell_pow_comm}\end{align} \end{proposition}\begin{proof}By direct computation, we can verify the following algebraic relations for Hilbert ladders: \begin{align}
\zeta_\ell\eta_m+\eta_\ell\zeta_m={}&\eta_{\ell +m},\label{eq:eta_ell_add}\\\zeta_\ell\zeta_m-\eta_{\ell}\eta_{m}={}&\zeta_{\ell +m}.\label{eq:zeta_ell_add}
\end{align}

For $ \ell=1$, the Feynman rules in \eqref{eq:H_zeta_ell_eta_ell} reproduce \eqref{eq:HT_IKsgn}--\eqref{eq:HT_KK} in Lemma \ref{lm:HT_I_K}. Assuming that the Hilbert transform formula $ \HT(\zeta_\ell)=\eta_\ell$ holds up to a certain integer  $\ell$, then $\HPB(\zeta_\ell,\zeta_1)$ and \eqref{eq:eta_ell_add}--\eqref{eq:zeta_ell_add} bring us $ \HT\eta_{\ell+1}=-\zeta_{\ell+1}$, which is equivalent to  $ \HT(\zeta_{\ell+1})=\eta_{\ell+1}$. By induction, this proves  \eqref{eq:H_zeta_ell_eta_ell} in its entirety.

For $\ell=1$, the statement in \eqref{eq:zeta_ell_pow_comm} is just $ \HT(\zeta_1)=\eta_1$. For $\ell=2$, the relation  $ \HT(\zeta_2\varpi_1)=\eta_2\varpi_1$ is effectively proved in \eqref{eq:HT_IIKK_KKKK_p1}. Now suppose that up to a certain integer  $\ell$,  we have  $ \HT(\zeta_\ell \varpi_{ j})=\eta_\ell \varpi_j, \forall j\in\mathbb Z\cap[0,\ell) $. From  $ \ParH(\zeta_\ell \varpi_j,\eta_2 \varpi_1)$ and $ \zeta_{\ell+2} \varpi_{j+1}=(\zeta_\ell \varpi_j)(\zeta_2\varpi_1)-(\eta_{\ell}\varpi_j)(\eta_{2}\varpi_1)=-(\zeta_\ell \varpi_j)\HT(\eta_2\varpi_1)-(\eta_{2}\varpi_1)\HT(\zeta_\ell \varpi_{ j})$, we know that $ \int_{\mathbb R}\zeta_{\ell+2} (x)x^{j+1}\D x=0, \forall j\in\mathbb Z\cap[0,\ell)$. According to the moment formula in \eqref{eq:HT_pow}, it follows that   $ \HT(\zeta_{\ell+2} \varpi_{ j+2})=\eta_{\ell+2} \varpi_{j+2}, \forall j\in\mathbb Z\cap[0,\ell) $. Additionally,   $ \ParH(\zeta_\ell ,\eta_2 )$ and $\zeta_{\ell+2} =\zeta_\ell \zeta_2-\eta_{\ell}\eta_{2}=-\zeta_\ell \HT\eta_2-\eta_{2}\HT\zeta_\ell$ tell us that  $ \int_{\mathbb R}\zeta_{\ell+2} (x)\D x=0 $. Thus $ \HT (\zeta_{\ell+2} \varpi_{ 1})=\eta_{\ell+2} \varpi_{1}$ is also true. This proves that    $ \HT(\zeta_{\ell+2} \varpi_{ j})=\eta_{\ell+2} \varpi_{j}, \forall j\in\mathbb Z\cap[0,\ell+2) $, so  \eqref{eq:zeta_ell_pow_comm} follows by induction.          \end{proof}

\begin{theorem}[B$^3$G sum rule]\label{thm:B3G} We have\begin{align}
Z_{2n,n-2k}:={}&\sum_{m=0}^{\lfloor n/2\rfloor}(-1)^m{n\choose{2m}}\int_0^\infty [\pi I_0(t)]^{n-2m}[K_0(t)]^{n+2m}t^{n-2k}\D t=0\label{eq:Z_vanish}
\end{align}for all integer pairs $(n,k)$ meeting the requirements  $n\geq2k\geq2$. \end{theorem}
 \begin{proof} Simply notice that binomial expansion leads us to \begin{align}
\zeta_n=\sum_{m=0}^{\lfloor n/2\rfloor}(-1)^m{n\choose{2m}}\left(\iota\sgn\right)^{n-2m}\kappa^{n+2m},
\end{align}      and  \eqref{eq:zeta_ell_pow_comm} implies that $ \int_{\mathbb R}  \zeta_n(x)x^j\D x=0$ for  $ j\in\mathbb Z\cap[0,n-1)$.  \end{proof}\begin{remark}Since Hilbert inversion brings us   $ \HT(\eta_\ell \varpi_{ j})=-\zeta_\ell \varpi_j=\varpi_j\HT\eta_\ell$ for $ j\in\mathbb Z\cap[0,\ell)$, we must have $ \int_{\mathbb R}  \eta_\ell(x)x^j\D x=0$ for  $ j\in\mathbb Z\cap[0,\ell-1)$, as well. Excluding from these vanishing identities the trivial cases where the integrands are  odd functions over $\mathbb R$, we arrive at another family of cancelation formulae, conjugate to the B$^3$G sum rule:\begin{align}
Y_{2n,n-2k}:=\sum_{m=1}^{\lfloor n/2\rfloor+1}(-1)^{m}{n\choose{2m-1}}\int_0^\infty [\pi I_0(t)]^{n-2m+1}[K_0(t)]^{n+2m-1}t^{n-2k-1}\D t=0,\label{eq:Y_seq}
\end{align}where    $n-1\geq2k\geq2$.
 \eor\end{remark}\begin{remark}Before writing this paper, we constructed alternative (and actually simpler) proofs of the B$^3$G sum rule \eqref{eq:Z_vanish} and its conjugate \eqref{eq:Y_seq}, via a vanishing contour integral\begin{align}
\int_{-i\infty}^{i\infty}z^{m}[H_0^{(1)}(z)H_0^{(2)}(z)]^n\D z=0,\quad m\in\mathbb Z\cap[0,n-1),\label{eq:BBBG_Hankel}
\end{align}where  $ H_0^{(1)}$, $ H_0^{(2)}$ are  cylindrical Hankel functions  and the contour can be closed to the right \cite[cf.][\S7.2]{Watson1944Bessel}. In spite of this complex analytic shortcut, we encountered ``infrared divergence'' (singular behavior of the integrand   as $ |z|\to\infty$) in \eqref{eq:BBBG_Hankel}, when we attempted to raise the power $m$ to higher values and calculate  $ \IKM(a,b;m)$ for  $ m+1\geq (a+b)/2$, which is the situation occurring in the Broadhurst--Mellit integer sequence (see \S\ref{subsec:B_int_seq} below). As we will soon see, the Hilbert transform method will succeed where  the contour integration approach fails.    \eor\end{remark}\subsection{Broadhurst--Mellit integer sequence (Crandall numbers)\label{subsec:B_int_seq}}Let \begin{align}A(n):=\left(2/\pi\right)^42^{2n-1}[\pi^{2}\IKM(3,5;2n-1)-\IKM(1,7;2n-1)],\quad \forall n\in\mathbb Z_{>0}\end{align} be the sequence defined  in \eqref{eq:An}. Despite (erstwhile) conjectural integrality, this now becomes  entry   A262961 in the On-line Encyclopedia of Integer Sequences (OEIS) \cite[][under the title ``Crandall numbers'', in homage  to Richard E. Crandall (1947--2012), who made seminal contributions to the computation of Bessel moments]{OEISA262961}, along with the following comments (at the time of writing, in  April 2017):
\begin{quote}\textit{``Anton Mellit and David Broadhurst define the sequence to be the} \textbf{\textit{`round'}} (emphasis added) \textit{of the integral, with the conjecture that this rounding is exact. No one seems to know how to prove that any of the integrals gives a rational number, let alone an integer.''}\end{quote}

By \eqref{eq:Y_seq}, we see that $ Y_{8,2}=0$  implies  the following vanishing identity\begin{align}
A(1):=\frac{32}{\pi^{4}}\int_0^\infty\left\{ [\pi I_0(t)]^2-[K_0(t)]^2 \right\}I_0(t)[K_0(t)]^5t\D t={}&0,
\end{align}  which was first conjectured by Broadhurst in  \cite[][(148)]{Broadhurst2016}, after collaboration with Mellit.  The rest of the sequence $ A(n+1),n\in\mathbb Z_{>0}$ in Conjecture \ref{conj:B_int_seq} is characterized by the theorem below.\begin{theorem}[Broadhurst--Mellit integer sequence, \textit{viz.}\ Crandall numbers]For all $n\in\mathbb Z_{>1}$, we have\begin{align}A(n):=\left(\frac{2}{\pi}\right)^4
\int_0^\infty\left\{[\pi I_0(t)]^2 - [K_0(t)]^2\right\} I_0(t)[K_0(t)]^5\,(2t)^{2n-1} \D t
\in\mathbb Z_{>0}.
\end{align}Furthermore, the following explicit formula holds for $ n\in\mathbb Z_{>0}$:\begin{align}
A(n+1)={}&\frac{1}{2^{4(n-1)}}\sum_{m=1}^{n} \sum_{\ell=1}^{m}\sum_{k=1}^{\ell} \frac{ [(2 n-2m)!]^3}{ [(n-m)!]^4}\frac{ [(2 m-2\ell)!]^3}{ [(m-\ell)!]^4}\frac{ [(2 \ell-2k)!]^3}{ [( \ell-k)!]^4}\frac{ [(2k-2)!]^3}{ [(k-1)!]^4}.\label{eq:An_explicit}
\end{align}\end{theorem}\begin{proof}We recall the following  Bessel moments from \cite[][(7)]{Bailey2008} (see also \cite{Ouvry2005,BorweinBruno2007}):\begin{align}
\int_0^\infty [K_0(t)]^{2}t^{2n}\D t=\frac{\pi ^2 [(2 n)!]^3}{4^{3 n+1} (n!)^4},\quad \forall n\in\mathbb Z_{\geq0},
\end{align}which enable us to compute the following Hilbert transforms with the aid of \eqref{eq:HT_pow} and \eqref{eq:HT_KK}:\begin{align}
\HT(\kappa^2 \varpi_{2n})={}& \varpi_{2n}\HT(\kappa^2)-\frac{1}{\pi}\sum_{k=0}^{2n-1}  \varpi_k\int_{-\infty}^\infty \xi^{2n-1-k}[\kappa(\xi)]^{2}\D\xi\notag\\={}& \iota\kappa \varpi_{2n}\sgn-\frac{\pi}{2}\sum_{m=1}^{n}  \frac{ [(2 n-2m)!]^3}{4^{3 (n-m)} [(n-m)!]^4}\varpi_{2m-1},\quad \forall n\in\mathbb Z_{>0}.
\end{align}Spelling out  $ \ParH(-3\iota^2\kappa^4\varpi_{1}+\kappa^6\varpi_{1},\kappa^2\varpi_{2n})$ [\textit{i.e.}\ $ \ParH(\eta_3\varpi_1,\kappa^2\varpi_{2n})$], we see that $A(n+1),\forall n\in\mathbb Z_{>0}$ is expressible as a finite sum of  Bessel moments involving only six Bessel functions:\begin{align}
A(n+1)=\frac{2^{2(n+1)}}{\pi^{4}}\sum_{m=1}^{n}  \frac{ [(2 n-2m)!]^3}{4^{3 (n-m)} [(n-m)!]^4}\int_0^\infty \{3[\pi I_{0}(t)]^2-[K_0(t)]^2\}[K_0(t)]^4t^{2m}\D t.
\end{align}By $\ParH(2\iota \kappa^3\sgn,\kappa^2\varpi_{2m})$, we arrive at further reduction:\begin{align}&
\int_0^\infty \{3[\pi I_{0}(t)]^2-[K_0(t)]^2\}[K_0(t)]^4t^{2m}\D t\notag\\={}&\pi^{2}\sum_{\ell=1}^{m}  \frac{ [(2 m-2\ell)!]^3}{4^{3 (m-\ell)} [(m-\ell)!]^4}\int_0^\infty I_0(t)[K_0(t)]^3t^{2\ell-1}\D t.
\end{align}  With  $ \HT(\kappa^2 \varpi_1)=\iota\kappa \varpi_1\sgn$ and  $\ParH(\kappa^{2}\varpi_{1},\kappa^2\varpi_{2(\ell-1)})$, we evaluate the remaining Bessel moments involving four Bessel functions:\begin{align}
\int_0^\infty I_0(t)[K_0(t)]^3t^{2\ell-1}\D t=\frac{\pi^{2}}{4^{3\ell-1}}\sum_{k=1}^{\ell}  \frac{ [(2 \ell-2k)!]^3}{ [( \ell-k)!]^4}\frac{ [(2k-2)!]^3}{ [(k-1)!]^4}.\label{eq:IKKK_moments}
\end{align}
So far, we have verified  \eqref{eq:An_explicit}.

Next, we show that \begin{align}
\alpha_\ell:=\frac{1}{2^{4(\ell-1)}}\sum_{k=1}^{\ell}  \frac{ [(2 \ell-2k)!]^3}{ [( \ell-k)!]^4}\frac{ [(2k-2)!]^3}{ [(k-1)!]^4}\in\mathbb Z,\quad \forall \ell\in\mathbb Z_{>0}.
\end{align}Towards this end, we  first recall   from  \cite[][(55) and (56)]{Bailey2008} the following formula\begin{align}
\int_0^\infty I_0(t)[K_0(t)]^3t^{2\ell-1}\D t=\frac{\pi^{2}}{16}\left[ \frac{(\ell-1)!}{4^{\ell-1}} \right]^2\sum^{\ell-1}_{k=0}{\ell-1\choose k}^2{2(\ell-1-k)\choose\ell-1-k}{2k\choose k},
\end{align}which combines with  \eqref{eq:IKKK_moments} into\footnote{Na\"ively, upon observing that $ (n!)^2/2^n\in\mathbb Z,\forall n\in\mathbb Z_{\geq4}$ and $ D_n\in\mathbb Z,\forall n\in\mathbb Z_{\geq0}$, we obtain $ 2^{\ell-1}\alpha_\ell\in\mathbb Z,\forall \ell\in\mathbb Z_{>0}$, at best. The divisibility statement   $ 4^{n}\mid(n!)^2D_n,\forall n\in\mathbb Z_{\geq0}$ is thus deeper than these na\"ive observations. In our proof of the integrality  $ \alpha_\ell\in\mathbb Z,\forall \ell\in\mathbb Z_{>0}$, we need  Rogers' work on modular forms \cite{Rogers2009}, which in turn, was inspired by Bertin's studies of modular parametrizations for certain families of  Calabi--Yau manifolds \cite{Bertin2008}.}
\begin{align}\label{eq:a_l_Domb}
\alpha_\ell=\frac{[(\ell-1)!]^{2}}{4^{\ell-1}}\sum^{\ell-1}_{k=0}{\ell-1\choose k}^2{2(\ell-1-k)\choose\ell-1-k}{2k\choose k}=\frac{[(\ell-1)!]^{2}}{4^{\ell-1}}D_{\ell-1},
\end{align}where $D_j,j\in\mathbb Z_{\geq0}$ is called the $j$-th Domb number in combinatorics.
In \cite[][Theorem 3.1]{Rogers2009}, Rogers has shown that the following identity holds for $|u|$ sufficiently small:\begin{align}
_3F_2\left(\left.\begin{array}{c}
\frac{1}{3},\frac{1}{2},\frac{2}{3} \\1,1 \\
\end{array}\right|\frac{27u^{2}}{4(1-u)^{3}}\right)=(1-u)\sum_{n=0}^\infty \frac{D_n}{4^n}u^n.\label{eq:Rogers2009}
\end{align}Here, the  generalized hypergeometric series $ _pF_q$ is defined as \begin{align}{_pF_q}\left(\left.\begin{array}{c}
a_{1},\dots,a_p \\[4pt]
b_{1},\dots,b_q \\
\end{array}\right| x\right):=\sum_{n=0}^\infty\frac{(a_1)_n\cdots(a_p)_n}{(b_1)_n\cdots(b_q)_n}\frac{x^n}{n!},\end{align} where the Pochhammer symbol represents the rising factorial\begin{align*}(a)_{n}:=\begin{cases}1, & n=0 \\
a(a+1)\cdots(a+n-1), & n\in\mathbb Z_{>0} \\
\end{cases} ;\quad (a)_{n}:=\frac{\Gamma(a+n)}{\Gamma(a)},\quad a\notin\mathbb  Z_{\leq0},\end{align*}  for $ \{b_1,\dots,b_q\}\cap\mathbb  Z_{\leq0}=\emptyset$. In Rogers' identity \eqref{eq:Rogers2009}, the coefficient of $u^n$ for $n\in\mathbb Z_{>0}$ is equal to \begin{align}\frac{D_n}{4^n}-\frac{D_{n-1}}{4^{n-1}},\end{align}according to the right-hand side. Meanwhile, the Taylor expansion of the left-hand side reads\begin{align}
_3F_2\left(\left.\begin{array}{c}
\frac{1}{3},\frac{1}{2},\frac{2}{3} \\1,1 \\
\end{array}\right|\frac{27u^{2}}{4(1-u)^{3}}\right)=1+3\sum_{n=1}^\infty[(2n-1)!!]^2{3n-1\choose2n}\frac{1}{(n!2^n)^2}\frac{u^{2n}}{(1-u)^{3n}},
\end{align}where the double factorial is defined as $ (2n-1)!!:=(2n)!/(n!2^n)\in\mathbb Z$ for $n\in\mathbb Z_{>0}$. As we compute the contribution from the aforementioned Taylor coefficients to\begin{align}(n!)^{2}
\left(\frac{D_n}{4^n}-\frac{D_{n-1}}{4^{n-1}}\right)=\alpha_{n+1}-n^{2}\alpha_{n},
\end{align}we are gathering finitely many  summands, each of which is  an integer multiple of $ (k!!)^{2}\in\mathbb Z$ for a certain odd positive integer $k$ less than $n$.  Therefore, we have $ \alpha_1=1$ and  $ \alpha_{\ell+1}-\ell^{2}\alpha_{\ell}\in\mathbb Z$ for all $\ell\in\mathbb Z_{>0}$.

Finally, by discrete convolution, we see that   $A(n+1)$ is the coefficient of  $x^{n-1}$  in  the polynomial\begin{align}
\left(\sum_{\ell=1}^{n+1}\alpha_\ell x^{\ell-1}\right)^2=\left[\sum_{\ell=0}^n\frac{(\ell!)^2D_\ell }{4^{\ell}}x^\ell\right]^2\in\mathbb Z[x],
\end{align} so $A(n+1)\in\mathbb Z $ must hold.  \end{proof}
\begin{remark}More generally, we define the positive integer  $ \alpha_n^{[m]}$ as the  coefficient of $x^{n-1}$ in the polynomial $\left(\sum_{\ell=1}^{n+1}\alpha_\ell x^{\ell-1}\right)^m\in\mathbb Z[x]$, for every $m\in\mathbb Z_{>0}$. By repeated applications of the recursions for Hilbert ladders [see \eqref{eq:eta_ell_add} and \eqref{eq:zeta_ell_add}], we can show that \begin{align}
\alpha_n^{[m]}=\frac{4}{\pi^{2m+1}}\sum_{\ell=1}^{m+1}(-1)^{\ell-1} {2m\choose{2\ell-1}}\int_0^\infty [\pi I_0(t)]^{2(m-\ell)+1}[K_0(t)]^{2(m+\ell)-1}(2t)^{2(n+m)-3}\D t.\label{eq:alpha_n_m}
\end{align} Thus, the Broadhurst--Mellit integer sequence $ A(1)=0,A(n+1)=\alpha_{n}^{[2]},n\in\mathbb Z_{>0}$ is just a special case within an infinite family of (linear combinations for) Bessel moments, preceded by \begin{align}
\alpha_{n}^{\vphantom{[1]}}\equiv\alpha_{n}^{[1]}=\frac{8}{\pi^{2}}\int_0^\infty I_0(t)[K_0(t)]^{3}(2t)^{2n-1}\D t\in\mathbb Z_{>0}
\end{align}and followed by\begin{align}
\alpha_{n}^{[3]}=\frac{8}{\pi^{6}}\int_0^\infty\left\{[\pi I_0(t)]^2 -3 [K_0(t)]^2\right\} \left\{3[\pi I_0(t)]^2 - [K_0(t)]^2\right\} I_0(t)[K_0(t)]^7\,(2t)^{2n+3} \D t\in\mathbb Z_{>0},
\end{align}for $ n\in\mathbb Z_{>0}$. \eor\end{remark}
\begin{remark}For $ \IKM(a,b;c)$ satisfying $c\equiv0\pmod2 $ and $ a+b\equiv2\pmod4$, we also have a family of  integer sequences that runs parallel to \eqref{eq:alpha_n_m}. With the observation that \begin{align}
\frac{ [(2 n)!]^3}{2^{4n}(n!)^4}=\frac{1}{2^{2n}}[(2n-1)!!]^{2}{2n\choose n}\in\frac{1}{2^{2n}}\mathbb Z,\quad \forall n\in\mathbb Z_{\geq0},
\end{align}we define $ \beta^{[m]}_n$ as the  coefficient of $x^{n-1}$ in the polynomial\begin{align}
\left[\sum_{\ell=0}^n\frac{(\ell!)^2D_\ell }{4^{\ell}}x^\ell\right]^{m-1} \sum_{k=0} ^n\frac{ [(2 k)!]^3}{2^{4k}(k!)^4}x^{k}\in\frac{1}{2^{2(n-1)}}\mathbb Z[x],
\end{align}  for every $m\in\mathbb Z_{>0}$. According to the discrete convolutions of Bessel moments that descend from the recursions for Hilbert ladders, we have \begin{align}
\beta^{[m]}_n=\frac{4}{\pi^{2m}}\sum_{\ell=1}^{m}(-1)^{\ell-1} {2m-1\choose{2\ell-1}}\int_0^\infty [\pi I_0(t)]^{2(m-\ell)}[K_0(t)]^{2(m+\ell-1)}(2t)^{2(n+m-2)}\D t\in\frac{1}{2^{2(n-1)}}\mathbb Z_{>0}
\end{align}for  $ m,n\in\mathbb Z_{>0}$. We can even combine the foregoing statements about $\alpha_n ^{[m]}$ and $ \beta_n^{[m]}$ into a more compact form, as follows:\begin{align}
\frac{2^{1+2(n-1)[1-(-1)^{M}]}}{\pi^{M+1}}\int_0^\infty \frac{[\pi I_0(t)+i K_{0}(t)]^{M}-[\pi I_0(t)-i K_{0}(t)]^{M}}{i}[K_0(t)]^{M}(2t)^{2n+M-3}\D t\in\mathbb Z_{>0}
\end{align}for all $ M,n\in\mathbb Z_{>0}$. This confirms a recent empirical observation by Broadhurst and Roberts \cite[][Conjecture 2]{Broadhurst2017Paris}. \eor\end{remark}

To summarize, we have demonstrated that many sum rules of Bessel moments indeed issue from the three fundamental laws \ref{itm:HT-1}--\ref{itm:HT-3} of  Hilbert transforms.  We remind our readers that the methods developed in this note do not exhaust all possible linear relations for Feynman diagrams in two-dimensional quantum field theory. For example, the sum rule      $ 9\pi^2\IKM(4,4;1)-14\IKM(2,6;1)=0$   \cite[conjectured in][(147)]{Broadhurst2016} is not a simple consequence of Hilbert transforms, and its recent verification  in \cite[][\S5.2]{WEF} involves heavy use of contour integrations over modular forms.
 \subsection*{Acknowledgments}In early 2017, I wrote up this paper  in Beijing, mostly drawing on my research notes prepared at Princeton during 2012 and 2013. I thank Prof.\ Weinan E for arranging my stays in Princeton and Beijing, as well as organizing a seminar on constructive quantum field theory at Princeton.

After completion of the initial draft of this article, I received from Dr.\ David Broadhurst his slides for recent talks  \cite{Broadhurst2017Paris,Broadhurst2017CIRM,Broadhurst2017Higgs} on Bessel moments, which set his conjectures in a wider context. I thank Dr.\ Broadhurst for his constant encouragements and incisive comments on this project.

I am indebted to an anonymous referee for thoughtful suggestions on improving the presentation of this paper.

In January 2013, I benefited from fruitful discussions with Prof.\ Jon Borwein on his previous contributions to Bessel moments and elliptic integrals; I was equally grateful to his friendly communications on my then-unpublished work related to Hilbert transforms.  I dedicate this work to his memory.

%\bibliography{HB}

%\bibliographystyle{plain}
%

\end{document}